\documentclass[12pt,draft]{amsart}
\usepackage{amssymb}
\usepackage{latexsym}
\usepackage{amsfonts}
\usepackage{amsmath}
\usepackage{color}

\usepackage{blindtext}

\newtheorem{theorem}{Theorem}[section]
\newtheorem{proposition}[theorem]{Proposition}
\newtheorem{lemma}[theorem]{Lemma}
\newtheorem{corollary}[theorem]{Corollary}

\newtheorem{conjecture}[theorem]{Conjecture}

\theoremstyle{definition}

\def\w{{\omega}}

\def\la{{\langle}}
\def\ra{{\rangle}}
\def\char{{\rm char \,}}

\def\b{\beta}

\renewcommand{\leq}{\leqslant}
\renewcommand{\geq}{\geqslant}

\newcommand{\aut}[1]{{\sf Aut}(#1)}

\newcommand{\lie}[1]{{\mathcal L}_{#1}}
\newcommand{\liealg}[2]{{\mathcal L}_{#1,#2}}
\newcommand{\asso}[1]{\mathcal{A}_{#1}}
\newcommand{\assoalg}[2]{{\mathcal A}_{#1,#2}}

\newcommand{\F}{\mathbb F}

\newcommand{\m}{\mathfrak m}
\newcommand{\n}{\mathfrak n}
\newcommand{\J}{\mathfrak J}

\newcommand{\ad}{{\rm ad}}

\newcommand{\gl}{{\mathfrak {gl}}}

\DeclareMathOperator{\ch}{char}
\DeclareMathOperator{\reg}{reg}

\begin{document}

\title[]{The isomorphism problem for universal enveloping algebras of
four-dimensional solvable Lie algebras}
\author{Jos\'e L. Vilca Rodr\'\i guez}
\address[Rodr\'\i guez]{Departamento de Matem\'atica\\
Instituto de Ci\^encias Exatas\\
Universidade Federal de Minas Gerais\\
Av.\ Ant\^onio Carlos 6627\\
Belo Horizonte, MG, Brazil. Email:
{\rm\texttt{vilca.rodriguez4@gmail.com}}}

\author{Csaba Schneider}
\address[Schneider]{Departamento de Matem\'atica\\
Instituto de Ci\^encias Exatas\\
Universidade Federal de Minas Gerais\\
Av.\ Ant\^onio Carlos 6627\\
Belo Horizonte, MG, Brazil. Email:
{\rm\texttt{csaba@mat.ufmg.br}}, URL:
{\tt
  {http://www.mat.ufmg.br/$\sim$csaba}}}

\author{Hamid Usefi}
\address[Usefi]{Department of Mathematics and Statistics,
Memorial University of Newfoundland,
St. John's, NL,
Canada, 
A1C 5S7}
\email{usefi@mun.ca}

\begin{abstract}
  This paper is a contribution to the isomorphism problem for universal
  enveloping algebras of finite-dimensional Lie algebras.
  We focus on
  solvable Lie algebras of small dimensions over fields of arbitrary
  characteristic.
 We  prove, over an arbitrary field,
  that the isomorphism type of a
  metabelian Lie algebra whose derived subalgebra has codimension one
  is  determined by its universal enveloping algebra. As an application
  of the results in this paper, we solve the isomorphism problem for
  solvable Lie algebras of dimension four over fields of characteristic
  zero and also point out the problems that occur in prime characteristic.
\end{abstract}

\date{16 March 2019}


\maketitle

\section{Introduction}\label{section1}

The main results of this paper make a contribution to the study of the isomorphism
problem for universal enveloping algebras of Lie algebras which asks if
an isomorphism $U(L)\cong U(H)$ between the universal enveloping algebras
of two Lie algebras $L$ and $H$ implies the isomorphism $L\cong H$.
 It is known that the answer to this question is, in general, ``no'' and
explicit examples of non-isomorphic
finite-dimensional Lie algebras with isomorphic universal enveloping algebras
were presented in~\cite{Usefi,Csaba}. On the other hand, all such
known examples are nilpotent and are defined over fields of characteristic $p$
with small  $p$ (in the sense that $p<\dim L$). Moreover, the known
results show that
the isomorphism problem does have a positive solution in several natural
and important classes
of finite-dimensional Lie algebras.
Indeed, over an arbitrary field of characteristic different from~2,
the isomorphism problem was solved for the class of $3$-dimensional
Lie algebras
by Malcolmson \cite{Malcolmson}
and by Chun, Kajiwara and Lee~\cite{Chun}.

The recent
paper~\cite{quasiisom} by Campos, Petersen, Robert-Nicoud, and Wierstra
announced the positive solution of the isomorphism problem for a large class
of finite-dimensional Lie algebras
including finite-dimensional nilpotent Lie algebras over fields
of characteristic zero. The proof of this result relies on the deep theory of
quasi-isomorphism and Koszul duality. The isomorphism problem for small-dimensional nilpotent Lie algebras over fields of arbitrary characteristic was
also addressed using computational techniques  by the last two
   authors~\cite{Csaba}. 
A detailed  treatment of the isomorphism problem can be found
in the article by Riley and the third author~\cite{Usefi}  and in the
survey~\cite{usefi-survey}.

It might be worth noting the analogies with a similar isomorphism  problem for integral group rings where a positive solution for the class of all nilpotent groups was given independently in \cite{RS} and
\cite{W}. There exists, however, a pair of non-isomorphic finite solvable groups of derived length 4 whose integral group rings are isomorphic  (see \cite{H}).  In view of Cartier-Kostant-Milnor-Moore theorem,  enveloping algebras and group algebras are building blocks of cocommutative Hopf algebras, so one might be able to view these problems in this broader context.

In analogy with group rings and in light of a recent positive solution in the class of finite-dimensional nilpotent Lie algebras, one may ask whether there is a counterexample of solvable Lie algebras for the isomorphism problem. So, in this paper we focus on solvable Lie algebras. We study  the isomorphism problem for finite-dimensional but not necessarily nilpotent Lie algebras and we mostly focus on techniques that are valid over fields of arbitrary characteristic. After the preliminary
Section~\ref{section:prelim} in which we summarize the required tools for the
rest of the paper, 
in Section~\ref{sec:metabelian}, we prove the following result concerning
a class of metabelian Lie algebras over arbitrary fields.
This can be viewed as an adaptation of the corresponding theorem on group rings 
by Whitcomb~\cite{Whitcomb} (see also  \cite[Theorem  9.3.13]{Polcino}).

\begin{theorem}\label{PropIdeal}
 Let $L=M\rtimes \langle x\rangle$ and $H=N\rtimes\langle y\rangle$ be finite-dimensional Lie algebras over an arbitrary field, where $M$ and $N$ are ideals of $L$ and $H$, respectively. Suppose that $\alpha: U(L)\rightarrow U(H)$ is an algebra isomorphism such that $\alpha(MU(L))=NU(H)$. Then, $L/M'\cong H/N'$.  
\end{theorem}

The other main result of Section~\ref{sec:metabelian} is the following
theorem which is obtained using the 
theory developed  by Elashvili and Ooms~\cite{Ooms2,Ooms1} concerning the
relationship between  the center $Z(U(L))$ and the Frobenius semiradical of~$L$.

\begin{theorem}\label{th:lprimezl}
  Suppose that $L$ and $H$ are  finite-dimensional metabelian Lie algebras over a field $\F$ of characteristic zero, and that $L=(L'+Z(L))\rtimes\langle x\rangle$ and $H=(H'+Z(H))\rtimes\langle y\rangle$. If $U(L)\cong U(H)$, then $L\cong H$. 
\end{theorem}

Given a Lie algebra $L$, an abelian ideal $M$ of $L$ and a certain maximal ideal $\m$ of $U(L)$, we construct in
Section~\ref{sec:ltilde} a Lie algebra $\liealg{M}{\m}$ and an associative algebra $\assoalg{M}{\m}$ using sections of $U(L)$. Choosing
$M$ and $\m$ appropriately, $\liealg{M}{\m}$ and $\assoalg{M}{\m}$ will be invariant under certain automorphisms of $U(L)$ 
(Theorem~\ref{lemma6} and Corollary~\ref{cor:isom}).
The algebras $\liealg{M}{\m}$ and $\assoalg{M}{\m}$ are explicitly computable
for finite-dimensional Lie algebras (Section~\ref{sec:examp}) and $\liealg{M}{\m}$
often turns out to be isomorphic to $L$ (Corollary~\ref{corollary2}).
Hence, for given Lie algebras $L$ and $H$ with abelian ideals
$M$ and $N$ and adequate maximal ideals of $U(L)$ and $U(H)$, respectively, computing $\liealg{M}{\m}$ and $\liealg{N}{\n}$ or $\assoalg{M}{\m}$ and $\assoalg{N}{\n}$ can lead to verifying
that $U(L)\ncong U(H)$. The Lie algebra $\liealg{M}{\m}$ and the associative algebra $\assoalg{M}{\m}$ are explicitly calculated
for several 4-dimensional solvable Lie algebras in Section~\ref{sec:examp}.

Our results allow us to  prove the following theorem which gives a
positive solution for the isomorphism problem for universal enveloping algebras
of 4-dimensional solvable Lie algebras over fields of characteristic zero.

\begin{theorem}\label{theorem1}
  Let $L$ and $H$ be solvable Lie algebras of dimension at most $4$ over a field of characteristic zero. If the universal enveloping algebras
  $U(L)$ and $U(H)$ are isomorphic, then  $L$ and $H$ must also be isomorphic.  
\end{theorem}

The proof of Theorem~\ref{theorem1} is presented in Section~\ref{section:proof} and  relies on de Graaf's classification
of 4-dimensional solvable Lie algebras~\cite{deGraaf} (see also
Section~\ref{section:proof}). First, we divide the class of 4-dimensional
solvable Lie algebras into 6 groups in such a way that two Lie algebras
in distinct groups cannot have isomorphic universal enveloping algebras
by the results of the aforementioned article~\cite{Usefi}.

The use of the Frobenius semiradical is the main reason why we could not state
Theorem~\ref{theorem1} in an arbitrary characteristic. Indeed, while Theorem~\ref{PropIdeal} and the constructions of $\liealg{M}{\m}$ and $\assoalg{M}{\m}$ do not require conditions
on the characteristic, the results that rely on the Frobenius semiradical
are only valid in characteristic zero. We conjecture that
Theorem~\ref{theorem1} holds also in positive characteristic,
and in Conjecture~\ref{conj}, we state three statements whose validity would
imply this claim, but the
verification
of these statements is beyond the means of the current paper.

\subsection*{Acknowledgments}
The research presented in this paper formed a part of the first author's
  PhD project under the supervision by the second author. Rodr\'\i guez was
  supported by PhD scholarships awarded by CAPES and CNPq.
  Schneider acknowledges
  the support of the CNPq grants {\em Universal} (project no.: 421624/2018-3)
  and {\em Produtividade em Pesquisa} (project no.: 308773/2016-0)
  and he is also grateful to the Department of Mathematics and Statistics
  of the Memorial University of Newfoundland for its hospitality during his
  visit in November, 2016.
  Usefi is  supported by the grant RGPIN: 2019-05650 from the Natural Sciences and Engineering Research Council of Canada and is grateful to   Departamento de Matem\'atica of
  Universidade Federal de Minas Gerais for its hospitality during his visit in February 2015.

\section{Preliminaries}\label{section:prelim}

  
Let $L$ be a Lie algebra over a field $\F$ and denote by $U(L)$ the
universal enveloping algebra  of $L$.  
Suppose that  $L$ is finite-dimensional and let  $\{x_1,\ldots,x_d\}$ be a basis of $L$. 
By the Poincar\'e--Birkhoff--Witt Theorem (see \cite[Theorem, p.~92]{Humphreys}),
the monomials of the form $x_1^{a_1}\cdots x_d^{a_d}$, with $a_i\geq 0$ for all $i$, form a basis for $U(L)$. 
In particular, the linear subspace generated by the $x_i$ is a Lie subalgebra of $U(L)$ which can be identified with the original Lie algebra $L$. Hence the associative algebra $U(L)$ is also generated by $L$. 
 The augmentation map $\varepsilon_L:U(L)\rightarrow\F$ is the unique algebra homomorphism that extends the map $x_i\mapsto 0$ for all $i$. The kernel of $\varepsilon_L$ is denoted by 
 $\w(L)$ and is referred to as the {\em augmentation ideal} of $U(L)$. The augmentation ideal $\omega(L)$ is the unique maximal ideal of $U(L)$ that contains $x_i$ for all $i$. An easy application of the Poincar\'e--Birkhoff--Witt Theorem shows that $\w(L)=LU(L)=U(L)L$. 
In $U(L)$, the $n$-th power $\omega(L)^n$ is an ideal which is denoted by $\omega^n(L)$. It is clear that the ideals $\omega^1(L)=\omega(L)$, $\omega^2(L)$,
$\omega^3(L),\ldots$ form a descending chain.

If $S$ is a subalgebra of $L$, then a basis 
of $S$ can be extended to a basis  of $L$, and so,  in virtue of the Poincar\'e--Birkhoff--Witt Theorem, $U(S)$ can be considered as a subalgebra of $U(L)$ and $\omega(S)$ can be considered
as a subalgebra of $\omega(L)$.

The following lemma collects some important properties of $U(L)$ and $\omega(L)$ that will be used in the paper. For the proof of parts (1) and (2) see \cite[Proposition 2.2.14]{Dixmier}, for part (3) see \cite[Proposition 6.1(1)]{Usefi}. The $i$-th term of the lower central series of a Lie algebra $L$ is denoted by $\gamma_i(L)$ starting
with $\gamma_1(L)=L$ and $\gamma_2(L)=L'$.

\begin{lemma}\label{lemmadixmer}
Let $L$ be a Lie algebra,   $M$  an ideal of $L$ and $S$ a subalgebra of $L$. Then
\begin{enumerate}
    \item  The right ideal $MU(L)$ of $U(L)$ generated by $M$ coincides with the left ideal $U(L)M$ of $U(L)$ generated by $M$. Hence $MU(L)=U(L)M$ is a two-sided ideal. 
    \item The homomorphism
      $U(L)\to U(L/M)$ induced by the  projection $\pi:L\to L/M$ is surjective with kernel $U(L)M$.
     \item $\omega(S)\cap\omega^n(S)\omega(L)=\omega^{n+1}(S)$; hence, $L\cap \omega^n(S)\omega(L)=\gamma_{n+1}(S)$. In particular $L\cap \w(M)\w(L)=L\cap M\w(L)=M'.$
\end{enumerate}
\end{lemma}


The  following  elementary result that we shall frequently use
can be quickly derived from  Lemma~\ref{lemmadixmer}(1).

\begin{lemma}\label{lemma7}
The ideal $L'U(L)$ coincides with the two-sided ideal of $U(L)$ generated by $\{ [a,b]:=ab-ba \mid  a, b \in U(L)\}$. Consequently, $L'U(L)$ is invariant under $\aut {U(L)}$.
\end{lemma}



It is known that two non-isomorphic Lie algebras may have isomorphic
universal enveloping algebras; examples are presented in  \cite{Usefi,Csaba}. Nevertheless, several properties of a Lie algebra
are determined by the isomorphism type of its    enveloping algebra. 
 For instance, it is a well known that if $L$ is a finite-dimensional Lie algebra, then the (linear) dimension of $L$ coincides with the Gelfand-Kirillov dimension of $U(L)$ (see \cite[8.1.15(iii)]{McConnell}); thus  $\dim L$ is determined only by the isomorphism type of $U(L)$.
 A detailed study of some properties of Lie algebras that are determined by the isomorphism type of its    universal enveloping algebra is given in  \cite{Usefi}; in the following lemma 
 we give a summary of the ones that we use in this paper.

\begin{lemma}\label{lemmausefi2}
Let $L$ and $H$ be  finite-dimensional Lie algebras and $\alpha: U(L)\rightarrow U(H)$ an algebra isomorphism. Then the following hold.
\begin{enumerate}
    \item $\dim L=\dim H$ and $\dim L/L'=\dim H/H'$.
    \item If $L$ is nilpotent, then so is $H$. Moreover, in this case the nilpotency classes of $L$ and $H$ coincide.
    \item $L'/L''\cong H'/H''$; in particular, $\dim L'/L'' = \dim H'/H''$. 
    \item If $L$ is metabelian, then so is $H$.
    \item If $L$ is  solvable, then so is $H$. 
    \item If $M$ and $N$ are ideals of $L$ and $H$, respectively, such that $\alpha(MU(L))=NU(H)$, then $M/M'\cong N/N'$.  
    \end{enumerate}
\end{lemma}

The next lemma is a variation of  \cite[Lemma  2.1]{Usefi}.

\begin{lemma}\label{lem:alphabar}
Let $L$ and $H$ be Lie algebras and suppose that $\alpha:U(L)\rightarrow U(H)$ is an isomorphism. Then there exists an isomorphism $\overline \alpha:U(L)\rightarrow U(H)$ such that
\begin{enumerate}
    \item $\overline\alpha(\w(L))=\w(H)$; and
    \item if $x\in L$ such that $\alpha(x)\in\w(H)$, then $\overline\alpha(x)=\alpha(x)$.
\end{enumerate}
\end{lemma}

At the end of this preliminary section, let us review some facts concerning the Frobenius semiradical of a Lie algebra in characteristic  zero. This concept will be used in the proof of Theorem  \ref{theorem1}. A more complete treatment  can be found in \cite{Ooms2,Ooms1}.
We denote the dual of a vector space $V$ by $V^*$. Let $L$ be a Lie algebra over a field $\F$ of characteristic  zero and let $f \in L^*$. Then we define
$$L(f)=\{x\in L\mid f([x,y])=0 \text{ for all }y\in L\}.$$
It is easy to see that $L(f)$ is a Lie subalgebra of $L$ containing the center $Z(L)$ of $L$. 
Set $$
i(L)=\min_{f\in L^*}\dim L(f);
$$ the number $i(L)$ is referred to as the \textit{index} of $L$.  An element $f\in L^*$ is called \textit{regular} if $\dim L(f)=i(L)$, the set of regular elements is denoted by $L^*_{\reg}$. The set 

$$F(L)=\sum_{f\in L^*_{\reg}}L(f)$$
is called the {\em Frobenius semiradical} of $L$.
The Frobenius semiradical is a subalgebra of $L$, containing $Z(L)$, which is invariant
under $\aut L$.

The following result links the Frobenius semiradical of a Lie algebra defined over a field of characteristic  zero to the center $Z(U(L))$ of its    enveloping algebra. If $X$ is a subset or an element of a Lie algebra $L$, then $C_L(X)$ denotes the centralizer of $X$ in $L$.

\begin{lemma}\label{lemma8}
Let $L$ be a Lie algebra over a field of characteristic zero. Then the following statements hold. 
\begin{enumerate}
    \item  Let $M$ be a Lie subalgebra of $L$ of codimension one. Then $i(M)=i(L)+1$ if and only if $F(L)\subseteq M$.
    \item  If $x\in L$ is such that $C_L(x)$ has codimension one in $L$, then $i(C_L(x))=i(L)+1$.
    \item  $Z(U(L))\subseteq U(F(L))$.
\end{enumerate}
\end{lemma}

\begin{proof}
See  \cite[Proposition 1.6(4)]{Ooms2} for part (1), \cite[Proposition 1.9(1)]{Ooms2} for part (2) and \cite[Theorem 2.5(1)]{Ooms1} for part (3). 
\end{proof}

Part  (3) of Lemma~\ref{lemma8} gives a necessary condition for an element to be central in $U(L)$ and this will be used in the proof of Theorem  \ref{theorem1}. The containment stated in Lemma~\ref{lemma8}(3) is not valid in characteristic  $p$, since, in this case, there exists a non-zero polynomial expression $p(x)$, for each $x\in L$, such that  $p(x)\in Z(U(L))$; see \cite[Chapter VI, Lemma 5]{Jacobsonbook}. 

The following consequence of Lemma~\ref{lemma8} will be used in the
proof of Theorem~\ref{th:lprimezl}.

  \begin{lemma}\label{lem:lprimezl}
    Suppose that $L$ is a metabelian  and non-abelian Lie algebra
    over a field of characteristic zero
    such that $L'+Z(L)$ has codimension one in $L$. Then $F(L)\subseteq
    L'+Z(L)$.
  \end{lemma}
  \begin{proof}
    Set $M=L'+Z(L)$. Clearly, $M\leq C_L(L')$ and, in fact, as $L$ is non-abelian,
    $M= C_L(L')$.
Since $M$ has codimension one in $L$,
there is $x\in L$ such that $L=M\rtimes\left<x\right>$. Moreover,
there exists $y\in L'$ such that $[y,x]\neq 0$, and it follows that
$M=C_L(y)$.
Now, Lemma~\ref{lemma8} implies that $F(L)\subseteq M$.
\end{proof}

\section{Universal enveloping algebras of metabelian Lie algebras}
\label{sec:metabelian}

In this section we prove Theorem~\ref{PropIdeal}
which shows that the isomorphism problem for
universal enveloping algebras has a positive solution for a rather large class of metabelian Lie algebras. 
This theorem  was inspired by a corresponding result for group algebras.
In the proof we use a well-known fact that if $\varphi$ is an automorphism of a polynomial ring $\F[x]$, then $\varphi(x)=ax+b$ with $a,b\in\F$ and $a\neq 0$; see  \cite[Proposition 3.1]{Nagata}.

\begin{proof}[The proof of Theorem~\ref{PropIdeal}]
Invoking  Lemma~\ref{lem:alphabar}, we assume  without loss of generality that
 $\alpha(MU(L))=NU(H)$ and  $\alpha(\w(L))=\w(H)$. We claim that 
\begin{equation}\label{claimyz}
  \mbox{for every $z\in L$ there exists $y_z\in H$ such that
    $\alpha(z)\equiv y_z\pmod{N\w(H)}$}.
\end{equation}
Let us first verify claim  \eqref{claimyz} in the case when
$z=x$. The map $\alpha$ induces an isomorphism between $U(L)/MU(L)$ and $U(H)/NU(H)$. On other hand, by Lemma~\ref{lemmadixmer}(2),
$$ U(L)/MU(L)\cong U(L/M) \quad \mbox{ and } \quad U(H)/NU(H)\cong U(H/N),$$
and so we can view $U(L)/MU(L)$ and $U(H)/NU(H)$ as polynomial rings in the
variables $x+MU(L)$ and $y+NU(H)$, respectively. Thus, using the
remark preceding the proposition, we have $\alpha(x)+NU(H)=ay+b+NU(H)$ where $a,b\in\F$ and $a\neq 0$. Further, as we assume $\alpha(\w(L))=\w(H)$,
it follows that $b=0$. Hence, $\alpha(x)+NU(H)=ay+NU(H)$, and it follows that there exist $z_1,\ldots, z_n\in N$ and $u_1,\ldots,u_n\in U(H)$ such that 
\begin{eqnarray*}
\alpha(x)= ay +\sum_{i=1}^n z_iu_i.
\end{eqnarray*}
Denoting the augmentation map of $U(H)$ by $\varepsilon_H$, we have
\begin{eqnarray*}
\alpha(x)&=& ay +\sum_{i=1}^n z_i u_i= ay +\sum_{i=1}^n z_i(u_i-\varepsilon_H(u_i)+\varepsilon_H(u_i))\\&=& ay+\sum_{i=1}^n \varepsilon_H(u_i)z_i+\sum_{i=1}^n z_i(u_i-\varepsilon_H(u_i)).
\end{eqnarray*} 
Since $u_i-\varepsilon_H(u_i)\in\omega(H)$, it follows that 
\begin{equation}\label{eq4}
  \alpha(x)\equiv ay+\sum_{i=1}^n \varepsilon_H(u_i)z_i\pmod{N\w(H)}.
\end{equation}
Note that the element on the right-hand side of equation  \eqref{eq4}
lies in $H$.
Denoting this element by $y_x$, we obtain that  
\begin{equation}\label{eq41}
  \alpha(x)\equiv y_x\pmod{N\w(H)}.  
\end{equation}
Hence \eqref{claimyz} is valid for the element $x\in L$.
If $z\in M$, then $\alpha(z)\in NU(H)=N+N\w(H)$,
and so there is some $y_z\in N$ such that
$\alpha(z)\equiv y_z\pmod{N\w(H)}$. Therefore property  \eqref{claimyz} holds
also for elements of $M$. As $L=M+\left<x\right>$, claim  \eqref{claimyz}
must hold for all elements $z\in L$.

Now, using Lemma~\ref{lemmadixmer}(3) we have that 
$$\alpha (M\w(L))=\alpha(MU(L)\w(L))=\alpha(MU(L))\alpha(\w(L))=NU(H)\w(H)=N\w(H).$$
Therefore the assignment $z+M\w(L)\mapsto y_z+N\w(H)$
for all $z\in L$ defines
an injective Lie algebra homomorphism
$$
L/(L\cap M\w(L))\cong (L+M\w(L))/M\w(L) \rightarrow (H+N\w(H))/N\w(H)\cong
H/(H\cap N\w(H)).
$$
Applying Lemma~\ref{lemmadixmer}(3), we obtain that
$L/(L\cap M\w(L))\cong L/M'$ and, similarly, that $H/(H\cap N\w(H))\cong H/N'$.
Thus there exists an injective Lie algebra homomorphism
$L/M'\rightarrow H/N'$. On the other hand,
Lemma~\ref{lemmausefi2}(6) shows that $\dim M/M'=\dim N/N'$,
and since $\dim L/M=\dim H/N=1$
we can conclude that $\dim L/M'=\dim H/N'$. Now it follows that
the injective homomorphism $L/M'\rightarrow H/N'$ is an isomorphism. 
\end{proof}

\begin{corollary}\label{corollary5}
Suppose that $L$ and $H$ are as Theorem~\ref{PropIdeal} and that $\alpha: U(L)\rightarrow U(H)$ is an algebra isomorphism. Then the following are valid.
\begin{enumerate}
    \item If either  $M$ or $N$ is abelian, then $L\cong H$.
    \item If $M=L'$ and $N=H'$, then $L/L''\cong H/H''$. If, in addition,
      either $L$ or $H$ is
metabelian, then $L\cong H$. 
\end{enumerate}
\end{corollary}
\begin{proof}
  (1) By Lemma~\ref{lemmausefi2}(1),  $\dim L=\dim H$.
  Suppose  that $M$ is abelian. By Theorem~\ref{PropIdeal},  we have  $L\cong H/N'$, and hence
  $N'=0$. Thus $L\cong H$, as claimed.

  (2) Lemma~\ref{lemma7} implies that $\alpha (L'U(L))=H'U(H)$.
  Thus Corollary~\ref{corollary5}(1), with $M=L'$ and $N=H'$, gives us $L/L''\cong H/H''$. If, in addition, $L$, say, is metabelian, then so is $H$,
  by Lemma~\ref{lemmausefi2}(4). In particular, $L''=0$ and
  $H''=0$, which implies $L\cong H$.
\end{proof}

It is worth noting that in~\cite[Theorem 3.1(ii)]{Csaba} we can find examples of non-isomorphic
metabelian Lie algebras with isomorphic universal enveloping algebras, in which the derived subalgebras have codimension different from one.  These algebras
are nilpotent and are defined over fields of positive characteristic.

We end this section with the proof  of Theorem~\ref{th:lprimezl}.
Recall that the Frobenius semiradical $F(L)$ of a Lie algebra $L$ was
introduced in  Section~\ref{section:prelim}.

\begin{proof}[The proof of Theorem~\ref{th:lprimezl}]
First, we note that the theorem is true if either $L$ or $H$ is abelian. Therefore, we assume that $L$ and $H$ are non-abelian. Set $M=L'+Z(L)$ and $N=H'+Z(H)$. By Lemma~\ref{lem:lprimezl}, $F(L)\subseteq M$ and $F(H)\subseteq N$. Applying Lemma~\ref{lemma8}(3), we have the following chains of inclusions:
 \begin{equation}\label{eq:inclusions}
   Z(U(L))\subseteq U(F(L))\subseteq U(M)\quad
   \mbox{and}\quad Z(U(H))\subseteq U(F(H))\subseteq U(N).
   \end{equation}
 Let $\alpha:U(L)\rightarrow U(H)$ be an isomorphism and
 suppose using Lemma~\ref{lem:alphabar} that $\alpha(\w(L))=\w(H)$. We claim
 that $\alpha(MU(L))=NU(H)$. 
By Lemma~\ref{lemma7}, $$\alpha(L'U(L))=H'U(H)\subseteq NU(H),$$  and since $Z(L)\subseteq Z(U(L))$, it follows that 
 $$
 \alpha(Z(L))\subseteq Z(U(H)).
 $$
 Further,
 $Z(L)\subseteq \w(L)$, and hence we obtain from equation~\eqref{eq:inclusions} that
  $$
 \alpha(Z(L))\subseteq Z(U(H))\cap\w(H)\subseteq U(N)\cap\w(H)=\w(N)\subseteq NU(H).
 $$
 Thus $\alpha(Z(L)U(L))\subseteq NU(H)$,
 and $$
 \alpha(MU(L))=\alpha(L'U(L))+\alpha(Z(L)U(L))\subseteq NU(H).
 $$
 Applying the same argument to the isomorphism $\alpha^{-1}:U(H)\rightarrow U(L)$, we have that $\alpha^{-1}(NU(H))\subseteq MU(L)$. Therefore $\alpha(MU(L))=NU(H)$, and since $M$ and $N$ are abelian ideals of $L$ and $H$, respectively,
 Corollary~\ref{corollary5}(1) implies that $L\cong H$.  
\end{proof}

\section{Finite-dimensional automorphism invariants of  universal enveloping algebras}
\label{sec:ltilde}

The aim of this section is to present constructions that, for a Lie algebra $L$,
for an abelian ideal $M$ of $L$, and for a certain maximal ideal $\m$ of $U(L)$, outputs another Lie algebra $\liealg{M}{\m}$ and
an associative algebra $\assoalg{M}{\m}$
that are constructed using sections of $U(L)$; the algebras $\liealg{M}{\m}$ and $\assoalg{M}{\m}$  are invariant under certain automorphisms of $U(L)$. We will show that $\liealg{M}{\m}$ can be isomorphic to $L$, and in such cases the isomorphism type of $L$ can be determined by the isomorphism type of $U(L)$. The constructions are based on arguments presented by~\cite{Chun}.

Throughout this section, $L$ denotes a finite-dimensional Lie algebra over a field $\F$ and $M$ an abelian ideal of $L$.
Let $K$ denote the quotient $L/M$.
Lemma~\ref{lemmadixmer} implies that $MU(L)$ is a two-sided ideal of $U(L)$ and $U(K)\cong U(L)/MU(L)$.
Since $M$ is abelian, $K$ has
a well-defined action $\ad_{K,M}:K\rightarrow\gl(M)$ given by
$\ad_{K,M}(k+M)(a)=[k,a]$ for all $k+M\in K$ and $a\in M$.
Then the image $\ad_{K,M}(K)$ of $K$ under $\ad_{K,M}$ is a subalgebra of $\gl(M)$.

Let $\m$ be a  maximal two-sided ideal of $U(L)$ such that $U(L)/\m\cong\F$ and $MU(L)\subseteq\m$. We can take, for instance, $\m=\w(L)$.
Since $MU(L)\m=M\m$, $M\m$ is a two-sided ideal of $U(L)$.
Set $$\widetilde I=MU(L)/M\m,$$ and 
consider $\widetilde I$
as an abelian Lie algebra.  Since $M$ is abelian, by
\cite[Proposition 2.4]{Chun}, the map
\begin{equation}\label{eq1}
    \vartheta:x\mapsto x+M\m
\end{equation}
is an isomorphism of the
(commutative) Lie algebras $M$ and $ \widetilde I$.  
Define 
a homomorphism $\beta:U(L)/MU(L)\rightarrow \gl(\widetilde I)$ by setting
\begin{equation}\label{action}
\beta(w+MU(L))(a+M\m)=[w,a]+M\m=(wa-aw)+M\m
\end{equation}
for all $w\in U(L)$ and $a\in MU(L)$.

\begin{proposition}\label{prop:Ltilde}
Let $L$, $M$, $\m$,  $\widetilde I$ and $\beta$ be as above. Then the following are valid.
\begin{enumerate}
    \item The map $\beta:U(L)/MU(L)\rightarrow\gl(\widetilde I)$ is a well-defined homomorphism of Lie algebras.
    \item  Let $J$ be the unique Lie ideal of $U(L)$ that contains $MU(L)$ and such that $\ker\beta=J/MU(L)$. Considering $\widetilde I$ and $\widetilde K=U(L)/J$ as Lie algebras and setting $$
    \liealg{M}{\m}=\widetilde I\rtimes_\beta \widetilde K,$$ we have that $\liealg{M}{\m}$ is a finite-dimensional Lie algebra. Moreover, if $K$ is abelian, then $\liealg{M}{\m}$ is metabelian.
  \item Let $\J=J\cap\m$. Then $\J$ is a two-sided associative ideal of $U(L)$
    and $\J$ has codimension one in $J$. In particular, $\dim U(L)/\J=\dim U(L)/J+1$.
    \item If $\m=\w(L)$, then $L\cap J=C_L(M)$ and $C_L(M)U(L)\subseteq \J$.
    \item If $\m=\w(L)$, then there is an injective homomorphism  $\widetilde\chi$ from $M\rtimes L/C_L(M)$ to  $\liealg{M}{\m}$ which is an isomorphism if and only if $\dim L/C_L(M)=\dim \widetilde K$.
\end{enumerate}
\end{proposition}
\begin{proof}
(1)  We claim that $\beta$ above is well-defined. Suppose that $w \in U(L)$,  $a, v\in MU(L)$, $u\in M\m$ and recall that $MU(L)\subseteq \m$. Then
  $$
  [w+v,a+u]=[w,a]+[w,u]+[v,a]+[v,u].
  $$
  Now  $[w,u]=wu-uw\in M\m$, $[v,a]=va-av\in M\m$, and $[v,u]=vu-uv\in M\m$. Thus $[w+v,a+u]+M\m=[w,a]+M\m$. Since $\beta$ is
  induced by the adjoint action of $U(L)$ on the ideal $MU(L)$, we have that $\beta:U(L)/MU(L)\rightarrow\gl(\widetilde I)$ is a Lie algebra homomorphism.
  
  (2) Since $\beta$ is a Lie algebra homomorphism, it follows that $\liealg{M}{\m}$ is a Lie algebra. Note that $\beta$ induces an injective homomorphism
  $\widetilde K\rightarrow\gl(\widetilde I)$, and so $\widetilde K$ is finite-dimensional.
  Hence $\liealg{M}{\m}$ is finite-dimensional. Now, if $K=L/M$ is abelian then so is $U(K)\cong U(L)/MU(L)$,
  and consequently $\widetilde K$ is also abelian. In this case
  $\liealg{M}{\m}$ is the semidirect sum of two abelian Lie algebras, and so
  $\liealg{M}{\m}$ is metabelian. 
  
  (3) Since $\J$ is a Lie ideal of $U(L)$, it is enough show that $\J$ is a right ideal. Suppose that $w\in \J$ and let $u\in U(L)$. Then we have, for $a\in MU(L)$, that 
$$[wu,a]=[w,a]u+w[u,a]=[w,a]u+[w,[u,a]]+[u,a]w.$$
  Since $w\in J$, the first two summands lie in $M\m$, and since $w\in \m$ the last summand is in $M\m$. Thus $wu\in \J=J\cap \m$, and so $\J$ is a
  two-sided ideal of $\w(L)$. Since $1\in J$, we have that $J+\m=U(L)$, and so
  \[
  J/\J=J/(J\cap \m)=(J+\m)/\m=U(L)/\m\cong \F.
  \]
  Therefore $\J$ has codimension one in $J$.
  
(4) Suppose that $\m=\w(L)$ and let $x\in L\cap J$. Then, by the
definition of $J$, we have, for all $y\in M$, that
$$
[x,y]\in L\cap M\w(L)=M'=0$$
(see Lemma~\ref{lemmadixmer}(3) for the second equality).
Thus $x\in C_L(M)$, and so  $L\cap J \subseteq C_L(M)$. For the converse,
let $x\in C_L(M)$ and let $a=\sum_{i=1}^n x_iu_i$, where $x_i\in M$ and
$u_i\in U(L)$, be an arbitrary element of $MU(L)$. Since $[x,x_i]=0$
for all $i$, it follows that 
$$[x,a]=x\sum_{i=1}^n x_iu_i-\sum_{i=1}^n x_iu_ix=\sum_{i=1}^n x_ixu_i-
\sum_{i=1}^n x_iu_ix.$$
Since every summand of the right-hand side of the last equation lies in $M\w(L)$, we have that $x\in L\cap J.$ Thus $C_L(M)\subseteq L\cap J$ and
the equality $C_L(M)\subseteq L\cap J$ follows.

For the second statement, note that $C_L(M)$ is an ideal of $L$, and so $C_L(M)U(L)$ is a two-sided ideal of $U(L)$, by Lemma~\ref{lemmadixmer}. As $\J$ is a two-sided ideal of $U(L)$ which contains $C_L(M)$, it follows directly that $C_L(M)U(L)\subseteq \J$.

  (5) Note that $M$ is a $K$-module with the adjoint action $\ad_{K,M}$. Hence $M$ is also a module for the associative algebra $U(K)$. Now $U(K)\cong U(L)/MU(L)$, and so  $U(K)$ acts on the vector space $\widetilde I$ by the composition of the isomorphism $U(K)\rightarrow U(L)/MU(L)$ and the homomorphism defined in  \eqref{action}. Hence the isomorphic vector spaces $M$ and $\widetilde I$ are
both $U(K)$-modules considering $U(K)$ as a Lie algebra. Let
$\psi:U(K)\rightarrow\gl(M)$ and $\tilde\psi:U(K)\rightarrow\gl(\widetilde I)$
denote the corresponding homomorphisms. Recall that $\vartheta: M\rightarrow \widetilde I$ as defined in equation (\ref{eq1}) is a linear isomorphism.
If $\alpha\in\gl(M)$ then $\tilde \alpha=\vartheta\alpha\vartheta^{-1}$ is
a corresponding endomorphism of $\widetilde I$ and the map
$\alpha\mapsto\tilde\alpha$ is an
isomorphism $\gl(M)\rightarrow\gl(\widetilde I)$.

We claim that $\widetilde{\psi(k)}=\tilde\psi(k)$ for all $k\in K$.
In other words, under the bijection $\vartheta:M\rightarrow \widetilde I$, an element $k\in K$ induces the same transformation on 
$M$ and on $\widetilde I$.
Suppose that $k\in K$ and $a+M\m\in \widetilde I$. As noted above,
  the elements $x_i+M\m$ form a basis for $\widetilde I$, where the set
  $\{x_i\}_i$ is a basis for $M$. Hence we may assume without
  loss of generality that $a\in M$. Then
  \begin{align*}
  \tilde\psi(k)(a+M\m)&=
  [k,a]+M\m=\psi(k)(a)+M\m=
  \vartheta \psi(k)\vartheta^{-1}(a+M\m)\\&=\widetilde{\psi(k)}(a+M\m).
  \end{align*}
  Thus the claim holds.
  
  Define the map $\chi:M\rtimes K\rightarrow \lie M=\widetilde I\rtimes_{\beta} \widetilde K$,
  by $\chi(a,k+M)=(\vartheta(a),k+J)$.
 Clearly, $\vartheta:M\rightarrow\widetilde I$ is a linear isomorphism
 and the map $k+M\mapsto k+J$ is a well-defined
 linear 
 homomorphism $K\rightarrow U(L)/J$.
 It follows from the argument in the previous paragraph that $\chi$ is a homomorphism of Lie algebras. It is immediate verify that the kernel of the map  $k+M\mapsto k+J$ is equal to $C_L(M)$. Further, $\chi$ induces an injective homomorphism $\widetilde{\chi}:M\rtimes L/C_L(M)\to \widetilde{I}\rtimes_{\beta} \widetilde{K}$ which is an isomorphism if and only if $\dim L/C_L(M)=\dim \widetilde{K}$. 
\end{proof}

In Proposition~\ref{prop:Ltilde}(2) we defined the Lie algebra
\[
\liealg{M}{\m}=\widetilde I\rtimes_{\beta}\widetilde K
\]
and
we can also define the associative algebra
\begin{equation}\label{eq:UM}
\assoalg{M}{\m}=U(L)/\J
\end{equation}
where $\J$ is the two-sided ideal of $U(L)$ defined in Proposition~\ref{prop:Ltilde}(3).
If $\m=\w(L)$ then we simply write $\lie{M}$ and $\asso{M}$ instead of $\liealg{M}{\m}$ and $\assoalg{M}{\m}$, respectively. If $L$ is
a finite-dimensional Lie algebra, then both $\liealg{M}{\m}$ and $\assoalg{M}{\m}$ are
finite-dimensional and they will be explicitly computed for some
4-dimensional solvable Lie algebras in Section~\ref{sec:examp}.
%
Now the main result of this section states that the Lie algebra $\liealg{M}{\m}$
and the associative algebra $\assoalg{M}{\m}$
are preserved
by certain automorphisms of $U(L)$.

\begin{theorem}
\label{lemma6}
Let $L$ and $H$ be finite-dimensional Lie algebras over a field $\F$ and
let $M$ and $N$ be
abelian ideals of $L$ and $H$, respectively. Suppose that
$\alpha:U(L)\rightarrow U(H)$ is an algebra isomorphism such
that $\alpha(MU(L))=NU(H)$.
Let $\m$ be a maximal ideal of $U(L)$ such that
$U(L)/\m\cong \F$ and $M\subseteq \m$ and set $\n=\alpha(\m)$.
Then $\alpha$ induces an isomorphism between the Lie algebras
$\liealg{M}{\m}$ and $\liealg{N}{\n}$ defined in Proposition~\ref{prop:Ltilde}(2) and between
the associative algebras $\assoalg{M}{\m}$ and $\assoalg{N}{\n}$
defined in~\eqref{eq:UM}.
\end{theorem}
 
\begin{proof}
  Denote by $\widetilde I_1=MU(L)/M\m$ and $\widetilde I_2=NU(H)/N\mathfrak{n}$. Since $\alpha(MU(L))=NU(H)$ we have that $$\alpha(M\m)=\alpha(MU(L)\m)=\alpha(MU(L))\alpha(\m)=NU(H)\mathfrak{n}=N\mathfrak{n}.$$
  Thus $\alpha$ induces an isomorphism
   $\overline\alpha$ between $\widetilde I_1$ and $\widetilde I_2$.
  
  Let $J_1/MU(L)$ and $J_2/NU(H)$ be the kernels of the actions of  $U(L)/MU(L)$ and $U(H)/NU(H)$ on $\widetilde I_1$ and $\widetilde I_2$, respectively,
  defined by equation \eqref{action}. 
  We will show that $\alpha(J_1)=J_2$.
  By definition
  $$
  J_1=\{w\in U(L)\mid [w,a]\in M\m \mbox{ for all }a\in MU(L)\}.$$
  Suppose that $w\in J_1$.
  Then $[w,a]\in M\m$ for all $a\in MU(L)$. As
  $\alpha(M\m)=N\mathfrak{n}$, we have
  $[\alpha(w),\alpha(a)]=\alpha([w,a])\in N\mathfrak{n}$ for
  all $a\in MU(L)$.
  As $a$ runs through all elements of $MU(L)$,  $\alpha(a)$ runs through all elements of $NU(H)$, and so
  we have that $[\alpha(w),a]\in N\mathfrak{n}$ for all $a\in NU(H)$.
  Therefore $\alpha(w)\in J_2$, and so $\alpha(J_1)\subseteq J_2$.
  The inclusion $\alpha(J_2)\subseteq J_1$
  can be proved similarly. Therefore 
    \begin{equation}\label{eqJ}
  \alpha(J_1)=J_2.
  \end{equation}
    Setting $\J_1=J_1\cap\m$ and $\J_2=J_2\cap\n$, this implies that
    \[
    \alpha(\J_1)=\alpha(J_1\cap \m)=J_2\cap\n=\J_2,
    \]
    and so $\alpha$ induces an isomorphism between $\assoalg{M}{\m}=U(L)/\J_1$ and
    $\assoalg{N}{\n}=U(L)/\J_2$.
In addition, equation~\eqref{eqJ} also implies that
    the map
      $\widetilde\alpha:U(L)/J_1\rightarrow U(H)/J_2$ defined by 
  $\widetilde\alpha(w+J_1)=\alpha(w)+J_2$, $w\in U(L)$,
  is an isomorphism. 
  
  Let us complete the proof of the claim concerning the isomorphism between
  $\liealg{M}{\m}$ and $\liealg{N}{\n}$. Define the map $(\overline{\alpha},\widetilde \alpha):\liealg{M}{\m}\rightarrow \liealg{N}{\n}$, by
  \[(\overline{\alpha},\widetilde \alpha)(a+M\m,w+J_1)=(\overline{\alpha}(a+M\m),\widetilde\alpha(w+J_2)).\]
  We will show that $(\overline{\alpha},\widetilde \alpha)$ is an isomorphism.
  Since $\overline\alpha$ and $\widetilde\alpha$ are isomorphisms, we are only
  required to show that
  $$
  \widetilde\alpha([a+M\m,w+J_1])=[\alpha(a)+N\mathfrak{n},\alpha(w)+J_2]
  $$
  for all $a\in MU(L)$ and $w\in U(L)$.
  Let us compute
  \begin{eqnarray*}
    {[}\alpha(a)+N\mathfrak{n},\alpha(w)+J_2{]}&=&[\alpha(a),\alpha(w)]+N\mathfrak{n}\\
    &=&\alpha({[}a,w])+N\mathfrak{n}
  =\widetilde\alpha([a+M\m,w+J_1]){}.
  \end{eqnarray*}
  Hence $(\overline{\alpha},\widetilde \alpha)$ is an isomorphism, as claimed.
\end{proof}

Using Lemma~\ref{lemma7}, Lemma~\ref{lem:alphabar} and
Theorem~\ref{lemma6}, and taking $M=L'$, $N=H'$, $\m=\w(L)$, and
$\n=\w(H)$, we obtain the following corollary.

\begin{corollary}\label{cor:isom}
Let $L$ and $H$ be finite-dimensional Lie algebras over a field $\F$ such that
$L'$ and $H'$ are abelian. If $U(L)\cong U(H)$, then
$\lie{L'}\cong \lie{H'}$ and $\asso{L'}\cong \asso{H'}$.
\end{corollary}

The following results will be useful in the determination
of $\lie M$ and $\asso M$ for some 4-dimensional solvable Lie algebras
in Section~\ref{sec:examp}.

\begin{corollary}\label{corollary2}
 Let $L$ be a Lie algebra and let $M$ be an abelian ideal of $L$ containing $L'$. Suppose that $L=M\rtimes K$, where $\dim L/C_L(M)\geq \lfloor (\dim M)^2/4\rfloor+1$. Then, the following hold.
 \begin{enumerate}
 	\item $\lie{M}\cong M\rtimes L/C_L(M)$.
 	\item If $C_L(M)=M$, then $U(L)=J\oplus K$ (as vector spaces).
\end{enumerate} 	
\end{corollary}

\begin{proof}(1) 
By Proposition~\ref{prop:Ltilde}(5), it suffices to show that
$\dim L/C_L(M)=\dim \widetilde{K}$. Since  $L'\subseteq M$, the Lie algebras $K=L/M$ and $\widetilde K$ are abelian. By Proposition~\ref{prop:Ltilde}(4),
$J\cap L=C_L(M)$, and hence
\[
L/C_L(M)=L/(J\cap L)\cong (L+J)/J\leq U(L)/J\cong \widetilde K.
\]
Thus $\dim \widetilde K\geq \dim L/C_L(M)\geq \lfloor (\dim M)^2/4\rfloor+1$. On the other hand,  $\widetilde K$ is an abelian subalgebra of $\gl(\widetilde I)$ and the maximal dimension of an abelian subalgebra of $\gl(\widetilde I)$  is  $\lfloor (\dim M)^2/4\rfloor+1$ (see  \cite[Theorem 1]{Jacobson}). Thus $\widetilde{K}\leq \lfloor (\dim M)^2/4\rfloor+1$, and so $\dim \widetilde K=\dim L/C_L(M)$ must hold.

(2) Since $C_L(M)=M$, we obtain $J\cap L=M$
by Proposition~\ref{prop:Ltilde}(4).
This implies that
$$
J\cap K=J\cap L\cap K=M\cap K=0.
$$
On the other hand,
$\dim U(L)/J=\dim\widetilde K$ and we have seen in the
proof of statement  (1) that $\dim K=\dim\widetilde K$. Thus $J$ and $K$
are linear subspaces of $U(L)$ such that $J\cap K=0$ and
$\dim U(L)/J=\dim K$. As $K$ is finite-dimensional, the vector space direct
decomposition $U(L)=J\oplus K$ follows.
\end{proof}

The next lemma can be used to explicitly determine $\asso M$ for some Lie
algebras.

\begin{lemma}\label{lemmaU_M}
  Let $L$ be a finite-dimensional Lie algebra, suppose that $M$ is an abelian ideal of $L$ and let $x\in L$ be such that  $L=C_L(M)\rtimes \langle x\rangle$. We have that
\[
\asso{M}=\F[x]/(f(x))
\]
where $f(x)$ is the smallest degree polynomial such that $f((\ad\, x)|_M)=0$ and $f(0)=0$. 
\end{lemma}

\begin{proof}
Let $J$ and $\J$ be as in Proposition~\ref{prop:Ltilde}.  By Proposition~\ref{prop:Ltilde}(4), $C_L(M)\subseteq J$ and, clearly,
  $C_L(M)\subseteq \w(L)$. Thus $C_L(M)U(L)\subseteq\J$. Therefore
  \[
  \asso{M}=U(L)/\J\cong (U(L)/(C_L(M)U(L)))/(\J/C_L(M)U(L)).
  \]
  Furthermore, by Lemma~\ref{lemmadixmer}(2),
  \[
  U(L)/(C_L(M)U(L))\cong U(L/C_L(M))
  \]
  and, since $L/C_L(M)\cong \left<x\right>$, we have that $U(L/C_L(M))\cong \F[x]$.
  Therefore, the composition
  \[
  \F[x]\rightarrow U(L/C_L(M))\rightarrow (U(L)/(C_L(M)U(L)))/(\J/C_L(M)U(L))
  \rightarrow \asso{M}
  \]
  is a surjective homomorphism from $\F[x]$ to
  $\asso{M}$ with kernel $\F[x]\cap \J$. In particular $\asso{M}$ is a commutative
  algebra and 
  $\asso{M}=\F[x]/(\F[x]\cap\, \J)$. As $\F[x]$ is a principal ideal domain, the ideal $\F[x]\cap\, \J$ is generated by a smallest degree polynomial in $\F[x]\cap\, \J$. Given a polynomial $f(x)\in\F[x]$ we have that
  $f(x)\in \J=J\cap\, \w(L)$ if and only if $f(x)\in J$ and $f(0)=0$. Since $f(0)=0$, $f(x)=a_1x+a_2x^2+\cdots+a_kx^k$, and so $f(x)\in J$ if and only if 
	$$[f(x),y]\in M\w(L)$$
  for every $y\in M.$

  We claim that
	\begin{equation}\label{eqf(x)}
	  (\ad \,x^n)y\equiv (\ad x)^n\,y \pmod{M\w(L)}
          \mbox{ for all $y\in M$ and $n\geq 1$}.
	\end{equation}
	 In fact, the claim is clear for $n=1$. Assume that~(\ref{eqf(x)}) is valid for $n\geq 1$ and note that
	 \begin{multline*}
	 	[x^{n+1},y]=x[x^n,y]+[x,y]x^n\equiv x(\ad\,x^n)y\equiv x(\ad\,x)^ny \equiv (\ad\,x)^{n+1}y\mod M\w(L).
 	 \end{multline*}
 	Now claim~\eqref{eqf(x)} follows by induction. This also implies that 
	$$[f(x),y]\equiv f(\ad\,x)y\mod M\w(L).$$
	Thus, since $f(\ad\,x)y\in L$ for every $y\in M$, we have that $f(x)\in J$ if and only if $$f(\ad\,x)y\in M\w(L)\cap L=M'=0;$$
	see Lemma~\ref{lemmadixmer}(3).
	That is $f(x)\in J$ if and only if $f(\ad\,x|_M)=0$. 
\end{proof}

\begin{corollary}\label{corollaryU_M}
Let $L$ be a finite-dimensional Lie algebra. Suppose that $M$ is an abelian ideal of $L$ and $x\in L$ such that  $L=C_L(M)\rtimes \langle x\rangle$. Then $\lie{M}\cong M\rtimes L/C_L(M)$ if and only if $(\ad\, x|_M)^2=\lambda \,(\ad\, x)|_M$ for some $\lambda\in \F$.
\end{corollary}

\begin{proof}
  By Proposition~\ref{prop:Ltilde}(5),
  $\lie{M}\cong M\rtimes L/C_L(M)$ is equivalent to the condition $\dim \widetilde{K}=\dim L/C_L(M)=1$, which (by Proposition~\ref{prop:Ltilde}(3))
  amounts to $\dim \asso{M}=2$. By Lemma~\ref{lemmaU_M} $\dim \asso{M}=2$ if and only if the degree of the smallest degree polynomial such that $f(\ad\, x|_M)=0$ and $f(0)=0$ is 2. This condition is equivalent to the condition that
  $f(x)=x^2-\lambda\,x$ with some $\lambda\in\F$,
  which, in turn, amounts to $(\ad\,x|_M)^2=\lambda\,\ad\,x|_M.$
\end{proof}

\section{$\lie{M}$ and $\asso{M}$ for some  4-dimensional solvable Lie algebras}
\label{sec:examp}

In this section, we determine the Lie algebra $\lie{M}$
(defined in Proposition~\ref{prop:Ltilde})
and the associative algebra $\asso{M}$ (defined in~\eqref{eq:UM})
for several of the isomorphism classes of
solvable Lie algebras $L$ of dimension 4.
A classification of these Lie algebras was given by de Graaf  \cite{deGraaf}, we reproduce this classification in Section~\ref{section:proof}. In each of the following calculations we determine $\lie{M}$ and $\asso{M}$ taking $M=L'$.

\subsection{$\boldsymbol{L=M^3_0}$}\label{subsec:m3}
First we calculate $\lie M$ for the Lie algebra $L=M^3_0$ defined, over an arbitrary field $\F$, by the multiplication table
$$L=M^3_0=\langle x_1,x_2,x_3,x_4\mid [x_4,x_1]=x_1, [x_4,x_2]=x_3, [x_4,x_3]=x_3\rangle.$$
Note that only the non-zero products are displayed in the presentation of
$L$; for example $[x_1,x_2]=0$, but it is not explicitly stated.
We have that $M=\langle x_1,x_3\rangle$ and $C_L(M)=
\langle x_1,x_2,x_3\rangle$, thus $L=C_L(M)\rtimes\langle x_4\rangle$. Note that 
\begin{equation*}\label{eq6}
(\ad\,x_4)|_M=\begin{pmatrix}1 & 0 \\ 0 & 1\end{pmatrix}.
\end{equation*}
Hence, applying Lemma~\ref{lemmaU_M} and Corollary~\ref{corollaryU_M}, we have that $$\lie{M}\cong \langle x_1,x_3\rangle\rtimes \langle x_4 \rangle \quad \mbox{ and }\quad \asso{M}=\F[x_4]/(x_4^2-x_4),$$ since $f(x)=x^2-x$ is the smallest degree polynomial such that $f(\ad\,x_4|_M)=0$ and $f(0)=0$.

\subsection{$\boldsymbol{L=M^6_{0,b}}$}\label{subsec:m6}
Let $\F$ be a field and let us now calculate $\lie {M}$
and $\asso M$ for the following family of isomorphism classes of Lie algebras
$$L=M^6_{0,b}=\langle x_1,x_2,x_3,x_4\mid [x_4,x_1]=x_2, [x_4,x_2]=x_3, [x_4,x_3]=bx_2+x_3\rangle,$$ 
where $b\in \F$. In this family, $M^6_{0,b}\cong M^6_{0,c}$ if and only if $b=c$; see Group~5 in Section~\ref{section:proof}. 
We have that $M=\langle x_2,x_3\ra $ and $C_L(M)=\la x_1,x_2,x_3\ra$, thus $L=C_L(M)\rtimes \la x_4 \ra$. Note that 
\begin{equation*}
\ad\,x_4|_M=\begin{pmatrix}0 & b\\ 1 & 1\end{pmatrix}.
\end{equation*}

First suppose that $b=0$. Then we have that $((\ad\,x_4)|_M)^2=\ad\,x_4|_M$. Hence, applying Lemma~\ref{lemmaU_M} and Corollary~\ref{corollaryU_M}, we obtain 
$$\asso{M}=\F[x_4]/(x_4^2-x_4)\quad \mbox{ and } \quad \lie{M}\cong \langle x_2,x_3\rangle \rtimes \langle x_4\rangle.$$

Now we suppose that $b\neq 0$. The smallest degree polynomial $f(x)$ such that $f((\ad\,x_4)|_M)=0$ and $f(0)=0$ is $f(x)=x^3-x^2-bx$. Thus by Lemma~\ref{lemmaU_M},
\[
\asso{M}=\F[x_4]/(x_4^3-x_4^2-bx_4).
\]
Choosing the basis $\{\overline 1,\overline{x_4}, \overline{x_4^2-x_4}\}$ of $\asso{M}$, we have that a basis for $\widetilde K$ is $\{\overline{x_4}, \overline{x_4^2-x_4}\}$. Now, as $\{\overline{x_2},\overline{x_3}\}$ is a basis of $\widetilde I$ and the action of $\widetilde K$ on $\widetilde I$ is induced by the map $\beta$ in~(\ref{action}), we can calculate the brackets in $\lie{M}=\widetilde I\rtimes_\beta \widetilde K$. First we note that
$$
[x_4^2,x_2]=[x_4,x_2]x_4+x_4[x_4,x_2]=x_3x_4+x_4x_3\equiv
x_3x_4+[x_4,x_3]\equiv bx_2+x_3\pmod{M\w(L)}.
$$
and that
\begin{eqnarray*}
[x_4^2,x_3] &=& [x_4,x_3]x_4+x_4[x_4,x_3]=(bx_2+x_3)x_4+x_4(bx_2+x_3)\\&\equiv&
bx_4x_2+x_4x_3\equiv bx_3+bx_2+x_3=(b+1)x_3+bx_2\pmod{M\w(L)}.
\end{eqnarray*}
Thus $[x_4^2-x_4,x_2]\equiv bx_2\pmod{M\w(L)}$ and $[x_4^2-x_4,x_3]\equiv bx_3\pmod{M\w(L)}$. Therefore, $\lie{M}$ is the 4-dimensional Lie algebra spanned by $\{\overline{x_2},\overline{x_3}, \overline{x_4}, \overline{x_4^2-x_4}\}$ and non-zero brackets 
$$
[\overline{x_4},\overline{x_2}]=\overline{x_3},\ 
[\overline{x_4},\overline{x_3}]=b\overline{x_2}+\overline{x_3},\ 
[\overline{x_4^2-x_4},\overline{x_2}]=b\overline{x_2},\ 
[\overline{x_4^2-x_4},\overline{x_3}]=b\overline{x_3}.
$$

\subsection{$\boldsymbol{L=M^7_{0,b}}$}\label{subsec:m7}
Let $\F$ be a field and let us now calculate $\lie{M}$ and $\asso{M}$
for the following family of isomorphism classes of Lie algebra
$$
L=M^7_{0,b}=\langle x_1, x_2, x_3, x_4\mid [x_4,x_1]=x_2, [x_4,x_2]=x_3, [x_4,x_3]=bx_2 \rangle,
$$
where $b\in \F$. In this family $M^7_{0,b}\cong M^7_{0,c}$ if and only if there exists $\lambda\in \F^\times$ such that $b=\lambda^2\,c,$ see Section~\ref{section:proof}. 

We have that $M=\langle x_2,x_3\rangle$ and $C_L(M)=\langle x_1,x_2,x_3\rangle$, and thus $L=C_L(M)\rtimes \langle x_4\rangle$. Note that 
\begin{equation*}\label{eqm7}
(\ad\,x_4)|_M=\begin{pmatrix}0 & b \\ 1 & 0\end{pmatrix}.
\end{equation*}

Suppose that $b=0$; then $(\ad x_4|_M)^2=0$. Applying Corollary~\ref{corollaryU_M} and Lemma~\ref{lemmaU_M}, we have that 
$$\lie{M}\cong\langle x_2,x_3\rangle\rtimes\langle x_4\rangle \quad\mbox{ and }\quad \asso{M}=\F[x_4]/(x_4^2).$$ 

Now suppose that $b\neq 0$. Then the smallest degree monic polynomial $f(x)$ such that $f((\ad\,x_4)|_M)=0$ and $f(0)=0$ is $f(x)=x^3-bx$. Thus,
by Lemma~\ref{lemmaU_M},
\begin{equation}\label{eq:x3}
  \asso{M}=\F[x_4]/(x_4^3-bx_4).
\end{equation}
Choosing the basis $\{\overline{1},\overline{x_4},\overline{x_4^2}\}$ of $\asso{M}$, we have that $\widetilde K=\{\overline{x_4},\overline{x_4^2}\}$. Since $\{\overline{x_2}, \overline{x_3}\}$ is a basis of $\widetilde I$ and  the action of $\widetilde K$ over $\widetilde I$ is induced by the map $\beta$ in~(\ref{action}), we can calculate the brackets in $\lie{M}=\widetilde I\rtimes_\beta \widetilde K$. For this, note that 
$$[x_4^2,x_2]=[x_4,x_2]x_4+x_4[x_4,x_2]=x_3x_4+x_4x_3\equiv
x_3x_4+[x_4,x_3]\equiv bx_2\pmod{M\w(L)}, $$
and 
$$[x_4^2,x_3]=[x_4,x_3]x_4+x_4[x_4,x_3]=bx_2x_4+bx_4x_2\equiv
bx_2x_4+b[x_4,x_2]\equiv bx_3\pmod{M\w(L)}, $$
Therefore $\lie{M}$ is a 4-dimensional Lie algebra spanned by $\{\overline{x_2}, \overline{x_3}, \overline{x_4}, \overline{x_4^2}\}$ and non-zero brackets 
%
  $$
  [\overline{x_4},\overline{x_2}]=\overline{x_3},\
  [\overline{x_4},\overline{x_3}]=b\overline{x_2},\
  [\overline{x_4^2},\overline{x_2}]=b\overline{x_2},\
  [\overline{x_4^2},\overline{x_3}]=b\overline{x_3}.
  $$

  It is interesting to notice that in the family $M_{0,b}^7$,
  the isomorphism type of the original Lie algebra $L$ is not determined by the Lie algebra $\lie{M}$;
  it is, however, determined by the commutative algebra $\asso{M}$
  in~\eqref{eq:x3} as implied by the following lemma.

  \begin{lemma}\label{lem:comalg} Let $\F$ be a field
    and if $\mbox{char}\,\F=2$ then also assume that $\F$ is perfect.
    Let $b,c\in \F^\times$. Then $\F[x]/(x^3-bx)\cong \F[x]/(x^3-cx)$ if and only if
    there exists $\lambda\in\F^\times$ such that
    $c=\lambda^2b$.
\end{lemma}
  \begin{proof}
First, if $\lambda\in\F^\times$ such that $c=\lambda^2b$, then the automorphism
of $\F[x]$ induced by $x\mapsto \lambda^{-1} x$ takes $x^3-bx$ to
$\lambda^{-3}x-\lambda^{-1}bx$ and hence the image of the ideal $(x^3-bx)$ is
$(x^3-\lambda^2bx)=(x^3-cx)$. Therefore the quotients
$\F[x]/(x^3-bx)$ and $\F[x]/(x^3-cx)$ are isomorphic. This argument
implies the claim of the lamma for perfect fields of characteristic two.

Assume now $\mbox{char}\,\F\neq 2$ and
also that $\F[x]/(x^3-bx)$ and $\F[x]/(x^3-cx)$ are isomorphic.
Suppose first that $b\in \F^2$; let $b=\beta^2$ for some
$\beta\in\F^\times$. Then $x^3-bx=x(x-\beta)(x+\beta)$ and so,
by the Chinese Remainder Theorem, 
\[
\F[x]/(x^3-bx)\cong \F[x]/(x)\oplus \F[x]/(x-\beta)\oplus \F[x]/(x+\beta)
\cong \F\oplus\F\oplus \F.
\]
Since $\F[x]/(x^3-bx)\cong \F[x]/(x^3-cx)$, the decomposition
$\F[x]/(x^3-cx)\cong\F\oplus\F\oplus\F$ must hold which forces that
$x^2-c$ is reducible, which  implies that $c\in \F^2$. This argument shows
that if one of $c$ or $b$ is an element of $\F^2$, then both are, and in
this case the conclusion of the lemma holds with $\lambda=1$.

Thus, only remains to analyze the case $b,c\not\in \F^2$.  Then 
$$\F[x]/(x^3-bx)=\F[x]/(x)\oplus \F[x]/(x^2-b)$$
and 
$$\F[x]/(x^3-cx)=\F[x]/(x)\oplus \F[x]/(x^2-c)$$
are the decompositions of the respective algebras into direct sums
of simple algebras.
Since $\F[x]/(x^3-bx)\cong \F[x]/(x^3-cx)$, it follows that
\[
\F[x]/(x^2-b)\cong\, \F[x]/(x^2-c).
\]
Suppose that $\delta+\lambda\,\overline{x}\in \F[x]/(x^2-c)$
is the image of $\overline{x}\in \F[x]/(x^2-b)$ under an isomorphism $\alpha:\F[x]/(x^2-b)\to\F[x]/(x^2-c)$,
where $\delta,\lambda\in\F$ and necessarily $\lambda\neq 0$.
Then
\[
0=\alpha(\overline x^2-b)=(\delta+\lambda \overline x)^2-b=
\delta^2+2\delta\lambda \overline x+\lambda^2c-b.
\]
This implies that $\delta=0$ and that $\lambda^2c=b$ as required.
\end{proof}

\section{The Proof of Theorem \ref{theorem1}}\label{section:proof}

For the proof of Theorem~\ref{theorem1}, we use the classification of
the 4-dimensional solvable Lie algebras given by de Graaf  \cite{deGraaf}.
Since  $\dim L$ is determined by the isomorphism type of
$U(L)$ (Lemma~\ref{lemmausefi2}(1)),
and the isomorphism problem for    enveloping algebras of solvable Lie algebras of dimension less than or equal to three
was solved in  \cite{Chun} over an arbitrary field,
we concentrate on the case of solvable Lie algebras of dimension 4.
The first step is to divide the algebras into six groups in such a way
that two algebras in different groups cannot have isomorphic   
enveloping algebras. The second step will be to verify 
that two non-isomorphic Lie algebras inside the same group cannot have  isomorphic   
enveloping algebras. 

First we state the classification of 3-dimensional solvable Lie algebras.
Recall that only the non-zero brackets are displayed in the multiplication tables
of Lie algebras.

\begin{theorem}\label{th:3dim}
  Let $L$ be a 3-dimensional solvable Lie algebra over a field $\F$. Then
  $L$ is isomorphic to one of the following algebras:
  \begin{enumerate}
    \item $L^1=\left<x_1,x_2,x_3\right>$ (the abelian Lie algebra);
    \item $L^2=\left<x_1,x_2,x_3\mid [x_3,x_1]=x_1, [x_3,x_2]=x_2\right>;$
    \item $L^3_a=\left<x_1,x_2,x_3\mid [x_3,x_1]=x_2, [x_3,x_2]=ax_1+x_2\right>$ with some $a\in\F$;
    \item $L^4_a=\left<x_1,x_2,x_3\mid [x_3,x_1]=x_2, [x_3,x_2]=ax_1\right>$ with
      some $a\in\F$.
  \end{enumerate}
  If $L^i_j\cong L^u_v$ then $i=u$. Moreover $L^3_a\cong L^3_b$ if and only if
  $a=b$ and $L^4_a\cong L^4_b$ if and only if 
there is an $\alpha\in\F^*$ with $a=\alpha^2 b$.
\end{theorem}

In dimension four, we have  14 families of
isomorphism classes of solvable Lie algebras.
We group these Lie algebras into six groups in such a way that two Lie algebras
in distinct groups cannot have isomorphic  enveloping algebras.
We follow the notation introduced by de Graaf  \cite{deGraaf}.
Several of de Graaf's families, such as $M^3_a$, $M^{13}_a$, $M^6_{a,b}$
and $M^7_{a,b}$, are split between two groups.
Although Theorem  \ref{theorem1} is only
stated in characteristic  zero, we present the isomorphism types of
4-dimensional solvable Lie algebras over an arbitrary field and we 
consider the isomorphism problem of    enveloping algebras of
certain classes of such Lie algebras also over a field of prime characteristic.

The  4-dimensional solvable Lie algebras are presented in several tables,
each table corresponding to one group.
Each table contains five columns and the information contained in the columns
is as  follows.
\begin{description}
\item[Name] The name of the algebra in the notation of de Graaf  \cite{deGraaf}.
\item[relations] The non-trivial Lie brackets in the
  multiplication table of the algebra
  with respect to a basis $\{x_1,x_2,x_3,x_4\}$.
\item[char] Restrictions on the characteristic of the field
  over which the algebra is defined. If this entry is blank, then the
  algebra is defined over an arbitrary field.
\item[parameter] Restrictions on the parameters $a$, $b$ that may appear
  in the multiplication table of the algebra.
\item[isomorphism] The conditions under which isomorphism occurs between two
  algebras in the same family. If the family is described by a single parameter
  $a$, then the condition given in this entry is equivalent to the isomorphism
  $M^i_a\cong M^i_b$. If the family is described by two parameters $a$ and $b$, then
  the condition given in this entry is equivalent to the isomorphism
  $M^i_{a,b}\cong M^i_{c,d}$.
\end{description}

By the main result of  \cite{deGraaf}, there is no isomorphism between two
Lie algebras in different families, while two Lie algebras belonging to the
same family are isomorphic if and only if the condition in the
``isomorphism'' column of the corresponding line of the table holds.

\medskip

\noindent{\bf Group 1: The abelian Lie algebra.}
The first group contains only the abelian Lie algebra.

$$
\begin{array}{|c|c|c|c|c|}
  \hline
  \mbox{Name} & \mbox{relations} & \mbox{char} & \mbox{parameter} & \mbox{isomorphism}\\
  \hline
  M^1 & & & &\\\hline
  
  \end{array}
$$
\medskip

\noindent{\bf Group 2: The non-metabelian Lie algebras.}
A 4-dimensional non-metabelian solvable
Lie algebra over a field $\F$
is isomorphic to one of the following Lie algebras.

{\small$$
\begin{array}{|c|c|c|c|c|}
  \hline
  \mbox{Name} & \mbox{relations} & \mbox{char} & \mbox{parameter} & \mbox{isomorphism}\\
  \hline
  M^{12}&
  \begin{array}{c}{}[x_4,x_1]=x_1, [x_4,x_2]=2x_2\\
    {}[x_4,x_3]=x_3, [x_3,x_1]=x_2
    \end{array}
  & & &\\
  \hline
  M^{13}_a&
  \begin{array}{c}{}[x_4,x_1]=x_1+ax_3, \\{}[x_4,x_2]=x_2\\
    {}[x_4,x_3]=x_1, [x_3,x_1]=x_2\end{array}
    &  & a\in\F^* &

      a=b\\
  \hline
  M^{14}_a& \begin{array}{c}{}[x_4,x_1]=ax_3, [x_4,x_3]=x_1,\\
    {}[x_3,x_1]=x_2\end{array} & & a\in\F^* & a/b\in \F^2\\
  \hline
M^{11}_{a,b}& \begin{array}{c}{}[x_4,x_1]=x_1,[x_4,x_2]=bx_2\\
  {}[x_4,x_3]=(1+b)x_3, \\{}[x_3,x_1]=x_2\\
  {}[x_3,x_2]=ax_1\end{array} &
  2 & \begin{array}{c}
    a\in\F\setminus\{0\}\\b\in\F\setminus\{1\}\end{array} & 
    \begin{array}{c}
      a/c\in\F^2\\(\delta^2+(b+1)\delta+b)/c\in\F^2\\
     \mbox{where }\delta=(b+1)/(d+1)\end{array}\\
   \hline
\end{array}
$$}


\medskip

\noindent{\bf Group 3: Metabelian Lie algebras with one-dimensional derived subalgebra.}
Let $L$ be a 4-dimensional metabelian Lie algebra such that $\dim L'=1$. Then
$L$ is isomorphic to one of the following algebras.
{\small$$
\begin{array}{|c|c|c|c|c|}
  \hline
  \mbox{Name} & \mbox{relations} & \mbox{char} & \mbox{parameter} & \mbox{isomorphism}\\
  \hline
  M^4& [x_4,x_2]=x_3, [x_4,x_3]=x_3 & & & \\
  \hline
  M^5 & [x_4,x_2]=x_3 & & &\\
  \hline
\end{array}
$$}

\medskip

\noindent{\bf Group 4: Metabelian Lie algebras with three-dimensional
derived subalgebra.}
Suppose that $L$ is a 4-dimensional metabelian Lie algebra with $\dim L'=3$.
Then $L$ is isomorphic to one of the following Lie algebras.
{\small$$
\begin{array}{|c|c|c|c|c|}
  \hline
  \mbox{Name} & \mbox{relations} & \mbox{char} & \mbox{parameter} & \mbox{isomorphism}\\
  \hline
  M^2&
  \begin{array}{c}
    {}[x_4,x_1]=x_1, [x_4,x_2]=x_2,\\
    {}[x_4,x_3]=x_3\end{array} &  & & \\\hline
    M^3_a&
    \begin{array}{c}{}
      [x_4,x_1]=x_1, [x_4,x_2]=x_3,\\
      {}[x_4,x_3]=-ax_2+(a+1)x_3\end{array} & & a\in\F^* & a=b\\\hline
      M^6_{a,b}& \begin{array}{c}{}[x_4,x_1]=x_2, [x_4,x_2]=x_3,\\
        {}[x_4,x_3]=ax_1+bx_2+x_3\end{array} & &
\begin{array}{c}
  a\in\F^*\\
  b\in\F\end{array} &
  \begin{array}{c}
  a=c\\
  b=d\end{array}\\\hline
  M^7_{a,b}&[x_4,x_1]=x_2, [x_4,x_2]=x_3, [x_4,x_3]=ax_1+bx_2 & &
  \begin{array}{c}
  a\in\F^*\\
  b\in\F\end{array} &
  \begin{array}{c}
  a=\alpha^3c\\
  b=\alpha^2d\\
  \mbox{with $\alpha\in\F^*$}
  \end{array}\\
  \hline
\end{array}
$$}

\medskip

\noindent{\bf Group 5: Metabelian Lie algebras with two-dimensional
  derived subalgebra and non-trivial center.}
Let $L$ be a 4-dimensional metabelian Lie algebra with $\dim L'=2$ and
$\dim Z(L)\neq 0$. Then $L$ is isomorphic to one of the following algebras.
{\small$$
\begin{array}{|c|c|c|c|c|}
  \hline
  \mbox{Name} & \mbox{relations} & \mbox{char} & \mbox{parameter} & \mbox{isomorphism}\\
  \hline
  M^3_0& [x_4,x_1]=x_1, [x_4,x_2]=x_3, [x_4,x_3]=x_3& & & \\\hline
  M^6_{0,b}& [x_4,x_1]=x_2, [x_4,x_2]=x_3, [x_4,x_3]=bx_2+x_3 & & b\in\F &
  b=c\\\hline
  M^7_{0,b}& [x_4,x_1]=x_2, [x_4,x_2]=x_3, [x_4,x_3]=bx_2 & & b\in\F &
  \begin{array}{c}
    b=\alpha^2 c\\
    \mbox{with $\alpha\in\F^*$}
  \end{array}\\\hline
\end{array}
  $$}


\medskip

\noindent{\bf Group 6: Metabelian Lie algebras with two-dimensional
  derived subalgebra and trivial center.}
Let $L$ be a Lie algebra over a field $\F$ such that $\dim L'=2$ and $Z(L)=0$.
Then $L$ is isomorphic to one of the following Lie algebras.
{\small$$
\begin{array}{|c|c|c|c|c|}
  \hline
  \mbox{Name} & \mbox{relations} & \mbox{char} & \mbox{parameter} & \mbox{isomorphism}\\
  \hline
M^8& [x_1,x_2]=x_2, [x_3,x_4]=x_4 & & & \\\hline
M^9_a& \begin{array}{c}
  {}[x_4,x_1]=x_1+ax_2, \\{}[x_4,x_2]=x_1,\\
  {}[x_3,x_1]=x_1, \\{}[x_3,x_2]=x_2\end{array} &
&  \begin{array}{c}
      T^2-T-a\in\F[T]\\
      \mbox{is irreducible}\end{array}
  & \begin{array}{c}
      \mbox{if $\char\F\neq 2$}\\
      a+1/4=\alpha^2(b+1/4)\\
      \mbox{with $\alpha\in\F^*$}
  \\ \ \\
  \mbox{if $\ch\F=2$}\\
  T^2+T+a+b\in\F[T]\\
  \mbox{is reducible}\end{array}
 \\\hline
 M^{13}_0& \begin{array}{c}{}[x_4,x_1]=x_1, \\{}[x_4,x_2]=x_2,\\
   {}[x_4,x_3]=x_1, \\{}[x_3,x_1]=x_2\end{array}
& \neq 2 &  & \\\hline
   M^{10}_a & \begin{array}{c}{}[x_4,x_1]=x_2, [x_4,x_2]=ax_1, \\
     {}[x_3,x_1]=x_1, [x_3,x_2]=x_2\end{array}
     & 2 & \begin{array}{c}
       a\not\in\F^2\\
       \mbox{or $a=0$}\end{array}&
\begin{array}{c}
  Y^2+bX^2+a\\
  \mbox{has solution}\\
  (X,Y)\in\F\times\F\\
\mbox{with $X\neq 0$}\end{array}\\\hline
\end{array}
$$}

In \cite[Section 5.5]{deGraaf}, it is shown that if $\ch \F=2$
and $a\in \F^2\setminus\{0\}$ then $M^{10}_a\cong M^{13}_0$.
For this reason, the
Lie algebras $M^{10}_a$ are considered only for $a\notin \F^2$ or
 for $a=0$.
Similarly, we will consider the Lie algebra $M^{13}_0$ only in
characteristic different from two
and in the family
$M^{10}_a$ we allow the parameter $a=0$. 



\medskip

\noindent{\em The proof of Theorem  \ref{theorem1}.}
We prove Theorem  \ref{theorem1} by verifying a sequence of claims.
Before the first claim, recall that the Frobenius semiradical $F(L)$ of
a Lie algebra $L$ was defined in Section  \ref{section:prelim}.
\medskip

\noindent{\em Claim 1.}
Suppose that $L$ is a Lie algebra in Group  6
over a field $\F$ of characteristic  zero. Then $F(L)=0$ and $Z(U(L))=\F$.

\medskip

\noindent{\em The proof of Claim 1.}
Suppose that $L$ is a Lie algebra in Group  6 given by the
basis $\{x_1,x_2,x_3,x_4\}$ and let
$\{\varphi_1,\varphi_2,\varphi_3,\varphi_4\}$ be the dual basis of $L^*$.
Consider the following linear forms on $L$:
\begin{eqnarray*}
  L=M^8&:&\quad f=\varphi_2+\varphi_4;\\
  L=M^9_a&:&\quad f=\varphi_1+\varphi_2;\\
  L=M^{13}_0&:&\quad f=\varphi_2;\\
  L=M^{10}_a&:&\quad f=\varphi_2.\\
\end{eqnarray*}
Then, in each of the cases, the alternating bilinear form 
$$
B_f:L\times L\rightarrow \F,\quad (x,y)\mapsto f([x,y])
$$
is non-degenerate. This implies that the index $i(L)$,
introduced in Section  \ref{section:prelim}, is equal to
zero, and hence $F(L)=0$.
By Lemma~\ref{lemma8}, this shows that $Z(U(L))=\F$.\hfill$\Box$

\medskip

\noindent{\em Claim 2.} Suppose that $L$ and $H$ are 4-dimensional
solvable Lie algebras over a field of  characteristic  zero such that $L$ and
$H$ belong to distinct groups in the group division given above. Then
$U(L)\not\cong U(H)$. 

\medskip

\noindent{\em The proof of Claim 2.}
Lemma~\ref{lemmausefi2} implies that the isomorphism type of $U(L)$ determines
whether $L$ is abelian, whether $L$ is metabelian, and, if $L$ is metabelian,
it determines $\dim L'$. Hence we may assume without loss of generality that
$L$ is in Group  5 and $H$ is in Group  6. In this case, $Z(L)\neq 0$, and so
$Z(U(L))\neq \F$. On the other hand, Claim  1 implies that $Z(U(H))=\F$.
Therefore $U(L)\not\cong U(H)$.\hfill$\Box$

\medskip

\noindent
    {\em Claim 3.}
    If $L$ and $H$ are two Lie algebras in Group  2 over a field 
    $\F$ of characteristic different from 2
    such that
    $U(L)\cong U(H)$, then $L\cong H$.

\medskip

\noindent{\em The proof of Claim 3.}
First note that the Lie algebras $M^{11}_{a,b}$ only exist in characteristic two
and hence we can disregard them in this arguement.
Suppose that $L$ and $H$ are two
Lie algebras in Group  2 
    such that
    $U(L)\cong U(H)$ and assume that $\alpha:U(L)\rightarrow U(H)$ is
    an isomorphism. Note that $L'$ and $H'$ are ideals of codimension
    one in $L$ and in $H$, respectively. Therefore, there are elements
    $x\in L$ and $y\in H$ such that $L=L'\rtimes \left<x\right>$ and
    $H=H'\rtimes\left<y\right>$.  Thus
    Corollary  \ref{corollary5}(2) implies that
    $L/L''\cong H/H''$. We have that
    
\begin{eqnarray*}
  M^{12}/(M^{12})''&\cong& L^2;\\
  M^{13}_a/(M^{13}_a)''&\cong& L^3_a;\\
  M^{14}_a/(M^{14}_a)''&\cong& L^4_a;
\end{eqnarray*} where $L^2,L^3_a$ and $L^4_a$ are the Lie algebras that appear in Theorem~\ref{th:3dim}.
Hence, by Theorem  \ref{th:3dim}, if $L\not\cong H$, then
$L/L''\not\cong H/H''$. Thus it follows that  $L\cong H$.
\hfill$\Box$

\medskip

The following claim addresses Group  3.

\medskip

\noindent
    {\em Claim 4.}
    $U(M^4)\not\cong U(M^5)$ holds over an arbitrary field $\F$.
    \medskip
    
\noindent
    {\em The proof of Claim  4.}
This follows at once from  Lemma~\ref{lemmausefi2}(2), as
    $M^4$ is non-nilpotent, but $M^5$ is nilpotent.
    \hfill$\Box$

\medskip

\noindent{\em Claim  5.}
If $L$ and $H$ are Lie algebras in Group  4 such that $U(L)\cong U(H)$,
then $L\cong H$ holds over an arbitrary field $\F$.

\medskip

\noindent{\em The Proof of Claim  5.}
We have that $L=L'\rtimes\left<x_4\right>$ and $H=H'\rtimes\left<x_4\right>$
and that $L'$ and $H'$ are abelian. Hence if $U(L)\cong U(H)$, then
Corollary  \ref{corollary5} implies that $L\cong H$.
\hfill$\Box$

\medskip

\noindent{\em Claim  6.}
If $\ch\F=0$ and $L$ and $H$ are Lie algebras of  Group  5 such that $U(L)\cong U(H)$, then $L\cong H$.

\medskip

\noindent{\em The proof of Claim  6.}
For a Lie algebra $L$ in Group  5, let $\lie{L'}$ denote the Lie algebra
 constructed in Proposition~\ref{prop:Ltilde}
with respect to the abelian ideal $M=L'\unlhd L$
and the maximal ideal $\m=\w(L)\unlhd U(L)$. If $L$ and $H$ are two
Lie algebras in Group  5 and $U(L)\cong U(H)$, then by Corollary~\ref{cor:isom} $\lie{L'}\cong \lie{H'}$.


In Section  \ref{sec:examp}, we determined that $\lie{L'}$ is 3-dimensional if 
$L=M^3_0$, while it is 4-dimensional, if
$L=M^6_{0,b}$ or $L=M^7_{0,b}$ with $b\neq 0$. Moreover, the
algebras $\lie{L'}$ constructed for
$M^3_0$, $M^6_{0,0}$ $M^7_{0,0}$ are pairwise non-isomorphic.
Thus, 
 the possible non-isomorphic algebras with isomorphic   
 enveloping algebras must be isomorphic to  $M^6_{0,b}$ or $M^7_{0,c}$ with some
 $b,c\in\F$.
 Let us hence assume that $L$ and $H$ are Lie algebras
 that are isomorphic to either $M^6_{0,b}$ or $M^7_{0,c}$ with some
 $b,\ c\in\F$ such that $U(L)\cong U(H)$.
 Note that $L=M\rtimes\langle x_4\rangle$ and $H=N\rtimes\langle
 x_4\rangle$, where $M=L'+Z(L)$ and $N=H'+Z(H)$. Now
 Theorem~\ref{th:lprimezl} implies that $L\cong H$.
 \hfill$\Box$

\medskip

\noindent{\em Claim  7.}
If $L$ and $H$ are  Lie algebras in Group  6 over an arbitrary field,
such that $U(L)\cong U(H)$, then $L\cong H$. 

\medskip

\noindent{\em The proof of Claim  7.}
Suppose that $L$ and $H$ are as in the claim, and let $\alpha:U(L)\rightarrow
U(H)$ be an isomorphism.
We may assume by Lemma~\ref{lemma7} that $\alpha(\w(L))=\w(H)$.
Let us compute $\lie{L'}$ and $\lie{H'}$ as in
Proposition~\ref{prop:Ltilde} using the abelian ideals
$L'$ and $H'$, respectively.
Note that both $L$ and $H$ can be written as
$\langle x_1, x_2\rangle\rtimes \langle x_3,x_4\rangle$
where $\left<x_1,x_2\right>$ is the derived subalgebra.
Further, the second component of this semidirect sum decomposition acts
faithfully on the first component, and hence
Corollary~\ref{corollary2}(1) shows that
$\lie{L'}\cong L$ and $\lie{H'}\cong H$.
On the other hand, Theorem~\ref{lemma6} implies that
$\lie{L'}\cong \lie{H'}$. Thus $L\cong H$, as claimed.
\hfill$\Box$

\smallskip

Our arguments in the proof of Theorem~\ref{theorem1} work over
arbitrary characteristic, except for some parts of the proofs of
Claims~2, 3, and~6.  In Claim~2, the argument in prime characteristic
fails if $L$ is in Group~5 and $H$ is in Group~6. Hence in order to extend
Theorem~\ref{theorem1} to fields of arbitrary characteristic, one would
only need to prove the following statements.

\begin{conjecture}\label{conj}
    (1) If $L$ and $H$ are Lie algebras over a field of positive characteristic such that $L$ belongs to Group~5, $H$ belongs to Group~6 and $U(L)\cong U(H)$, then $L\cong H$.

(2)  If $L$ and $H$ are Lie algebras over a field
  of  characteristic two that belong to Group~2 such that
  $U(L)\cong U(H)$, then $L\cong H$.

  (3)  If $L$ and $H$ are Lie algebras over a field
  of positive characteristic that belong to Group~5 such that
  $U(L)\cong U(H)$, then $L\cong H$.
  \end{conjecture}

The calculations in Section~\ref{subsec:m7} and Lemma~\ref{lem:comalg} imply
that assertion~(3) of Conjecture~\ref{conj} is true if both $L$ and $H$ are of the form
$M_{0,b}^7$ for some $b\in\F$ and $\F$ is either of characteristic different from
two or perfect.

\end{document}